\documentclass{article}

\usepackage{fullpage}
\usepackage{authblk}

\usepackage{graphicx,amsmath,amssymb,amsfonts,amsthm,mathtools,hyperref,cleveref} 
\usepackage[dvipsnames]{xcolor}
\usepackage[numbers]{natbib}
\usepackage{csquotes}

\newtheorem{thm}{Theorem}
\newtheorem{cor}{Corollary}
\newtheorem{prop}{Proposition}

\newtheorem{defn}{Definition}
\newtheorem{lem}{Lemma}
\newtheorem{rmk}{Remark}

\newcommand{\st}[1]{\textup{st2}\left( #1\right)} 
\newcommand{\BE}[1]{T^\mathit{BE}_{#1}}
\newcommand{\fb}[1]{T^\mathit{fb}_{#1}}

\newcommand{\stBE}[1]{\st{\BE{#1}}}
\newcommand{\BT}[1]{\mathcal{BT}^\ast_{#1}}

\hypersetup{
    colorlinks=true,
    linkcolor=blue,
}

\title{On the maximum value of the stairs2 index}
\author[1]{Bryan Currie\thanks{\url{bc479@njit.edu}}}
\author[1]{Kristina Wicke \thanks{\url{kristina.wicke@njit.edu}}}

\affil[1]{Department of Mathematical Sciences \protect \\New Jersey Institute of Technology, Newark, NJ, USA}

\date{}

\begin{document}
\maketitle

\begin{abstract}
\noindent Measures of tree balance play an important role in different research areas such as mathematical phylogenetics or theoretical computer science. The balance of a tree is usually quantified in a single number, called a balance or imbalance index, and several such indices exist in the literature. Here, we focus on the stairs2 balance index for rooted binary trees, which was first introduced in the context of viral phylogenetics but has not been fully analyzed from a mathematical viewpoint yet. While it is known that the caterpillar tree uniquely minimizes the stairs2 index for all leaf numbers and the fully balanced tree uniquely maximizes the stairs2 index for leaf numbers that are powers of two, understanding the maximum value and maximal trees for arbitrary leaf numbers is an open problem in the literature. In this note, we fill this gap by showing that for all leaf numbers, there is a unique rooted binary tree maximizing the stairs2 index. Additionally, we obtain recursive and closed expressions for the maximum value of the stairs2 index of a rooted binary tree with $n$ leaves.
\end{abstract}

\section{Introduction} \label{Sec:Introduction}
Tree (im)balance indices play an important role in different research areas, ranging from phylogenetics and evolutionary biology, to cancer research, to theoretical computer science \cite{Fischer2023}. These indices quantify the balance or imbalance of a tree in a single number, and more than 20 such indices have been introduced in the literature over the years (see \citet{Fischer2023} for an overview). Many of them are widely used and well understood theoretically, while for others less is known regarding their mathematical properties. As indicated by \citet{Fischer2023}, an index for which there are still open questions regarding its combinatorial and statistical properties is the so-called \emph{stairs2} balance index. This index was originally introduced by \citet{Norstrom2012} in the context of phylodynamics and viral evolution, and was subsequently employed in a study by \citet{Colijn2014} in a study on the relationship of tree shape and transmission patterns underlying an outbreak of infectious diseases. 

While we give formal definitions in the next section, the stairs2 index intuitively measures the \enquote{staircase-ness} of a rooted binary tree by calculating the average ratio between the leaf numbers of the smaller and larger pending subtrees below an interior vertex across all inner vertices. It was recently classified by \citet{Fischer2023} as a balance index satisfying the following two criteria: (i) for all leaf numbers, the so-called caterpillar tree uniquely minimizes the stairs2 index, and (ii) for leaf numbers that are powers of two, the so-called fully balanced tree uniquely maximizes the stairs2 index. However, the maximum value and maximal trees for leaf numbers that are not powers of two have not been considered in the literature, and have been stated as open problems in \cite{Fischer2023}. 

In this manuscript, we fill these gaps in the literature. To be precise, we first prove that for all leaf numbers there is a unique rooted binary tree, which we call the \emph{binary echelon tree}, maximizing the stairs2 index. We then exploit the structure of the binary echelon tree to obtain both recursive and closed expressions for the maximum stairs2 index for a given leaf number $n$.

The remainder of this manuscript is organized as follows. In Section~\ref{Sec:Preliminaries} we review some general notation and definitions, before formally introducing the stairs2 index in Section~\ref{Sec:Stairs2}. In Section~\ref{Sec:Results} we then present our results concerning the maximal trees with respect to the stairs2 index and the maximum value itself. We end with some concluding remarks and directions for future research in Section~\ref{Sec:Discussion}.

\section{Preliminaries} \label{Sec:Preliminaries}
Before we can present our results, we need to introduce some general definitions and notation. We remark that throughout this paper, we mainly follow the terminology of \citet{Fischer2023}. Additionally, we remark that all logarithms are assumed to be base two logarithms. 

\paragraph{Rooted binary trees.}
A \emph{rooted tree} is a directed graph $T=(V(T), E(T))$, with vertex set $V(T)$ and edge set $E(T)$, containing precisely one vertex of in-degree zero, the root (denoted by $\rho$), such that for every $v \in V(T)$ there exists a unique path from $\rho$ to $v$, and such that there are no vertices with out-degree one. We use $V_L(T) \subseteq V(T)$ to refer to the leaf set of $T$ (i.e., $V_L(T) = \{v \in V(T): \text{out-degree}(v) = 0\}$), and we use $\mathring{V}(T)$ to denote the set of inner vertices of $T$ (i.e., $\mathring{V}(T) = V(T) \setminus V_L(T)$). Moreover, we use $n$ to denote the number of leaves of $T$, i.e., $n = \vert V_L(T) \vert$. Note that $\rho \in \mathring{V}(T)$ if $n\geq 2$. If $n=1$, $T$ consists of only one vertex, which is at the same time the root and its only leaf. For technical reasons, we will sometimes also consider the case $n=0$, in which case $T$ is the empty graph on zero vertices. 

A rooted tree is \emph{binary} if all inner vertices have out-degree two, and for every $n \in \mathbb{N}_{\geq 0}$, we use $\BT{n}$ to denote the set of (isomorphism classes) of rooted binary trees with $n$ leaves.
Since all trees in this paper are rooted and binary, we will often refer to them simply as \emph{trees}.

\paragraph{Depth and height.}
The \emph{depth} $\delta_T(v)$ of a vertex $v \in V(T)$ is the number of edges on the path from $\rho$ to $v$, and the height $h(T)$ of $T$ is the maximum depth of any leaf of $T$, i.e., $h(T) = \max\limits_{x \in V_L(T)} \delta_T(x)$.

\paragraph{Pending subtrees and decomposition of rooted binary trees.}
Given a tree $T$ and a vertex $v \in V(T)$, we denote by $T_v$ the pending subtree of $T$ rooted in $v$ and use $n_T(v)$ (or simply $n_v$ if there is no ambiguity) to denote the number of leaves in $T_v$. 
We will often decompose a rooted binary tree $T$ on $n \geq 2$ leaves into its two maximal pending subtrees rooted in the children of $\rho$. We denote this decomposition as $T = (T_1,T_2)$. We let $n_1$ and $n_2$ denote the number of leaves of $T_1$ and $T_2$, and, if not stated otherwise, assume that $n_1 \geq n_2$. 

\paragraph{Special trees.}
Finally, we need to introduce some specific families of trees that will be important throughout this manuscript. 

First, the \emph{fully balanced tree of height $h$} with $h \in \mathbb{N}_{\geq 0}$, denoted as $\fb{h}$, is the unique rooted binary tree with $n=2^h$ leaves in which all leaves have depth exactly $h$. Note that for $h \geq 1$ (i.e., $n \geq 2$) both maximal pending subtrees of a fully balanced tree are again fully balanced trees, and we have $\fb{h} = \left( \fb{h-1}, \fb{h-1} \right)$.

Second, the \emph{binary echelon tree} on $n$ leaves, denoted as $\BE{n}$, is recursively defined as follows: If $n=0$, $\BE{n}$ is the empty tree and if $n=1$, $\BE{n}$ consists of a single vertex. If $n \geq 2$, let $n/2 \leq k < n$ be a power of two. Then, $\BE{n} = \left(\fb{\log k}, \BE{n-k}\right)$, i.e., $\BE{n}$ consists of fully balanced subtree on $k$ leaves and a binary echelon tree on $n-k$ leaves (see Figure~\ref{fig:BE} for an illustration). Notice that $k$ is uniquely defined. Thus, the binary echelon tree is unique (up to isomorphism) for all $n$. Moreover, notice that $\BE{1}=\fb{0}$ and $\BE{2}=\fb{1}$. More generally, if $n=2^h$ for some $h \in \mathbb{N}_{\geq 1}$, then $k = 2^{h-1}$, and $\BE{n} = \left( \fb{h-1}, \fb{h-1}\right)$, i.e., $\BE{2^h}=\fb{h}$.

\begin{figure}[htbp]
    \centering
    \includegraphics[scale=0.125]{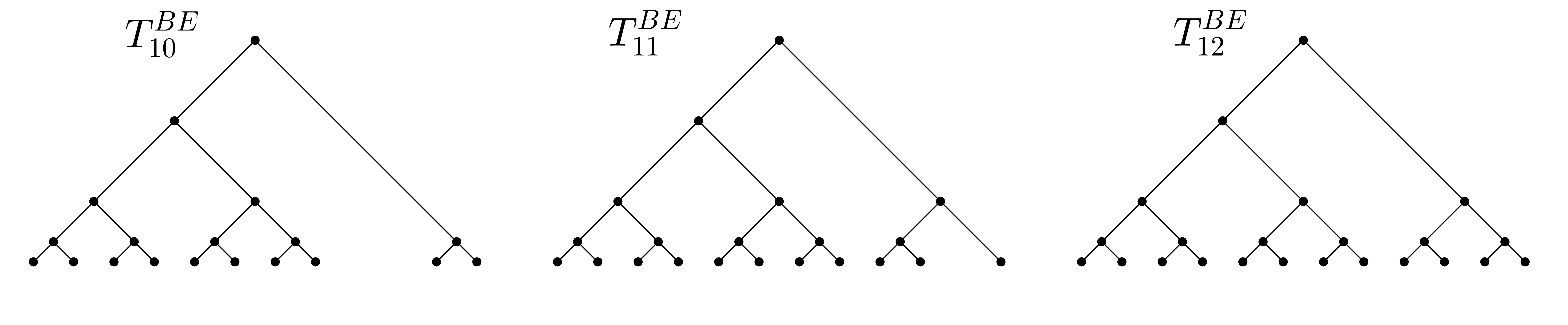}
    \caption{The binary echelon tree $\BE{n}$ for $n \in \{10,11,12\}$. Notice that in all three cases the larger subtree of $\BE{n}$ is the fully balanced tree of height three, $\fb{3}$.}
    \label{fig:BE}
\end{figure}

\section{The stairs2 index} \label{Sec:Stairs2}
We now introduce the central concept of this manuscript, namely the stairs2 index. This index was first introduced by \citet{Norstrom2012} (see also \citet{Colijn2014}) and is formally defined as follows.

\begin{defn}
    The \emph{stairs2 index} $\st{T}$ of a binary tree $T \in \BT{n}$ is defined as follows. If $n \in \{0,1\}$, we set $\st{T} \coloneqq 0$. If $n \geq 2$
    \begin{align*}
        \st{T} \coloneqq \frac{1}{n-1} \sum\limits_{v \in \mathring{V}(T)} \frac{\min \{n_{v_1}, n_{v_2}\}}{\max \{n_{v_1}, n_{v_2}\}},
    \end{align*}
    where $n_{v_1}$ and $n_{v_2}$ denote the number of leaves in the two maximal pending subtrees rooted in the children of $v$.
\end{defn}
In other words, the stairs2 index calculates the mean ratio between the leaf numbers of the smaller and larger pending subtrees over all inner vertices of a rooted binary tree. 

Note that the stairs2 index can be computed recursively, a property that was established by \cite{Fischer2023} and that will be used frequently throughout this manuscript.

\begin{lem}[adapted from {\citet[Proposition~23.25]{Fischer2023}}] \label{l:recursive}
   For every binary tree $T \in \BT{n}$ with $n \geq 2$ and standard decomposition $T=(T_1,T_2)$ such that $n_1 \geq n_2$, we have
   \begin{align*}
       \st{T} &= \frac{(n_1-1) \cdot \st{T_1} + (n_2-1) \cdot \st{T_2} + \frac{n_2}{n_1}}{n_1 + n_2 - 1}.
   \end{align*}
\end{lem}

\begin{rmk}
    We remark that Lemma~\ref{l:recursive} also holds if $T_2$ is the empty tree (in which case $T_1 = T, n_1=n, n_2=0$, and $\st{T_2}=0$), and we will use this fact freely throughout this manuscript. 
\end{rmk}

A second property established by \cite{Fischer2023} concerns the maximum value and maximal trees of the stairs2 index. 

\begin{prop}[{\citet[p.~319--320]{Fischer2023}}] \label{prop:fb}
    For every binary tree $T \in \BT{n}$ with $n \geq 2$, the stairs2 index fulfills $\st{T} \leq 1$. Moreover, for any given $n = 2^h$ with $h \in \mathbb{N}_{\geq 1}$, there is exactly one tree in $\BT{n}$ reaching this upper bound, namely the fully balanced tree $\fb{h}$.
\end{prop}

In particular, \citet{Fischer2023} only showed uniqueness of the maximal tree and tightness of the upper bound when the leaf number is a power of two. 
In the next section, we establish an analogous result for when the leaf number is not necessarily a power of two. We show that for all $n \in \mathbb{N}_{\geq 1}$, there is a unique binary tree maximizing the stairs2 index, namely the binary echelon tree $\BE{n}$. Moreover, we give recursive and closed expressions for the maximum value of the stairs2 index in this case.

\section{Results} \label{Sec:Results}
\subsection{Trees with maximum stairs2 index} \label{Subsec:MaxTrees}
The main aim of this section is to establish the following theorem concerning the maximal trees for the stairs2 index.

\begin{thm} \label{thm:BE-unique}
    For any given $n \in \mathbb{N}_{\geq 1}$, there is precisely one binary tree $T \in \BT{n}$ with maximum stairs2 index, namely the binary echelon tree $\BE{n}$.
\end{thm}

In order to prove this theorem, we require the following lemma, which states that if $T=(T_1,T_2)$ has maximum stairs2 index, its maximal pending subtrees $T_1$ and $T_2$ have maximum stairs2 index, too.

\begin{lem}\label{l:max-subtrees}
    Let $T=(T_1,T_2)$ be a rooted binary tree with $n \geq 2$ leaves. If $T$ has maximum stairs2 index on $\mathcal{BT}_n^\ast$, then $T_1$ and $T_2$ have maximum stairs2 index on $\mathcal{BT}_{n_1}^\ast$ and $\mathcal{BT}_{n_2}^\ast$, respectively.
\end{lem}

\begin{proof}
    For sake of contradiction, assume that $\st{T_1}$ is not maximal (the case when $\st{T_2}$ is not maximal is symmetrical). Then, there exists a $\widehat{T} \in \mathcal{BT}_{n_1}^\ast$ such that $\st{\widehat{T}} > \st{T_1}$. Consider the tree $\widetilde{T} = \left(\widehat{T}, T_2 \right) \in \BT{n}$ obtained by replacing $T_1$ by $\widehat{T}$ in $T$. Then, by Lemma~\ref{l:recursive},
    \begin{align*}
        \st{\widetilde{T}} &= \frac{(n_1-1) \cdot \st{\widehat{T}} + (n_2-1) \cdot \st{T_2} + \frac{\min \{n_1, n_2\}}{\max \{n_1, n_2\}}}{n_1 + n_2 - 1} \\
        &> \frac{(n_1-1) \cdot \st{T_1} + (n_2-1) \cdot \st{T_2} + \frac{\min \{n_1, n_2\}}{\max \{n_1, n_2\}}}{n_1 + n_2 - 1}
        = \st{T},
    \end{align*}
    implying that $\st{T}$ is not maximal. Thus, if $\st{T}$ is maximal, $\st{T_1}$ must be maximal, too.
\end{proof}

We are now in the position to prove Theorem~\ref{thm:BE-unique}. We will prove the statement by induction on $n$, using the fact that the two maximal pending subtrees of a maximal tree on $n$ leaves have to be maximal, too, and showing that only the partition of $n$ into $n_1$ and $n_2$ as induced by the binary echelon tree, yields the maximum stairs2 index. 

\begin{proof}[Proof of Theorem~\ref{thm:BE-unique}]
    We prove this statement by induction on $n$. For $n \in \{1,2,3\}$, $\BT{n}$ contains only one element (which is $\BE{n}$), and the statement trivially holds.

    Now assume that the claim holds for all binary trees with up to $n-1$ leaves, and consider a binary tree $(T_1,T_2)$ with $n \geq 3$ leaves. 

    If $n$ is a power of two, i.e., $n = 2^h$ for some $h \in \mathbb{N}_{\geq 1}$, the statement follows from Proposition~\ref{prop:fb} and the fact that $\BE{2^h}=\fb{h}$. 

    Thus, from now on we assume that $n$ is not a power of two, and therefore $\lfloor \log n \rfloor < \log n$.
    Let $T = (T_1, T_2) \in \BT{n}$ be a binary tree maximizing the stairs2 index, where $T_1$ and $T_2$ are binary trees with $n_1$ and $n_2$ leaves, respectively, and such that $n_1 \geq n_2$. By Lemma~\ref{l:max-subtrees}, $T_1$ and $T_2$ must have maximum stairs2 index, too, and so by the inductive hypothesis, they are binary echelon trees. In particular, $T = \left(\BE{n_1}, \BE{n_2} \right)$. To complete the proof, we need to show that $n_1 = 2^{\lfloor \log n \rfloor} \eqqcolon k$ (i.e., $k$ is the largest power of two less than $n$) and $n_2 = n - k \eqqcolon j$ (with $j < k$), which implies that $T = \left( \BE{k}, \BE{n-k} \right) = \left(\fb{\log k}, \BE{n-k} \right) = \BE{n}$.

    We prove this by showing that if $n_1 \neq k$ and $n_2 \neq j$ (with $k$ and $j$ as defined above), $T$ does not have maximum stairs2 index. We distinguish several cases:

    \begin{description}
        \item[Case 1:] First, assume that $n_1 = k+p$ and $n_2 = j-p$ with $0 < p < j$, i.e., 
        $$T = \left( \BE{k+p}, \BE{j-p} \right).$$ 
        We now consider the two maximal pending subtrees of $T$ and their stairs2 index more in-depth, exploiting the fact that they are both binary echelon trees.
        First, note that since $p < j < k$, we have $$\BE{k+p} = \left(\fb{\log k}, \BE{p}\right).$$ Using Lemma~\ref{l:recursive} and Proposition~\ref{prop:fb}, we thus obtain
        \begin{align}
            (k+p-1) \cdot \stBE{k+p} &= (k-1)\cdot \underbrace{\st{\fb{\log k}}}_{= 1} + (p-1)\cdot\stBE{p} + \frac{p}{k} \nonumber \\
            &= (k-1) +  (p-1)\cdot\stBE{p} + \frac{p}{k}. \label{BE-kp}
        \end{align}
        For $\BE{j-p}$, let $s$ denote the largest power of two less or equal to $j-p$ and let $t = j-p-s$. Then, $j-p = s+t$ with $t < s$ and 
        $$ \BE{j-p} = \left( \fb{\log s}, \BE{t} \right).$$
        Note that $t=0$ is possible, in which case $\BE{t}$ is the empty tree and $\stBE{t}=0$. Again, using Lemma~\ref{l:recursive} and Proposition~\ref{prop:fb}, we obtain 
         \begin{align}
            (j-p-1) \cdot \stBE{j-p} &= (s-1)\cdot \underbrace{\st{\fb{\log s}}}_{= 1} + (t-1)\cdot\stBE{p} + \frac{t}{s} \nonumber \\
            &= (s-1) +  (t-1)\cdot\stBE{t} + \frac{t}{s}. \label{BE-jp}
        \end{align}
        Using Lemma~\ref{l:recursive} and Equations~\eqref{BE-kp} and \eqref{BE-jp}, we thus obtain
        \begin{align}
            (n-1) \cdot \st{T} &= (k+p-1) \cdot \stBE{k+p} + (j-p-1) \cdot \stBE{j-p} + \frac{j-p}{k+p} \nonumber \\
            &= (k-1) +  (p-1)\cdot\stBE{p} + \frac{p}{k} \nonumber \\ 
            &\qquad  + (s-1) +  (t-1)\cdot\stBE{t} + \frac{t}{s} + \frac{j-p}{k+p}. \label{c1perturbed}
        \end{align}
        Now, consider the following tree
        $$ T' = \left(\fb{\log k}, \left(\BE{p}, \underbrace{\left(\fb{\log s}, \BE{t}\right)}_{= \BE{j-p}} \right) \right).$$ 
        We will show that $\st{T'} > \st{T}$ contradicting the maximality of $T$.
        Let $$\widehat{T} = \left(\BE{p}, \left(\fb{\log s}, \BE{t}\right) \right) = \left(\BE{p}, \BE{j-p} \right).$$ 
        Using Lemma~\ref{l:recursive}, Proposition~\ref{prop:fb}, and Equation~\eqref{BE-jp}, we obtain
        \begin{align} 
            (p+s+t-1) \cdot \st{\widehat{T}} &= (p-1) \cdot \stBE{p} + (s-1) +  (t-1)\cdot\stBE{t} + \frac{t}{s} + \frac{\min\{p,j-p\}}{\max\{p,j-p\}}, \label{BE-That}
        \end{align}
        and thus
        \begin{align}
            (n-1) \cdot \st{T'} &= (k-1) \cdot \underbrace{\st{\fb{\log k}}}_{=1} + (p+s+t-1) \cdot \st{\widehat{T}} + \frac{j}{k} \nonumber \\
            &= (k-1) + (p-1) \cdot \stBE{p} + (s-1) +  (t-1)\cdot\stBE{t} \nonumber \\
            &\qquad  + \frac{t}{s} + \frac{\min\{p,j-p\}}{\max\{p,j-p\}} + \frac{j}{k} \label{c1tprime}
        \end{align}
        Comparing Equations~\eqref{c1perturbed} and \eqref{c1tprime}, in order to show that $\st{T} < \st{T'}$, it suffices to show that
        \begin{align}
        \frac{p}{k} + \frac{j-p}{k+p} <  \frac{j}{k} + \frac{\min\{p,j-p\}}{\max\{p,j-p\}}. \label{ineqT2T1Case1}
        \end{align}
        To see that this inequality holds for $p \geq 1$, we distinguish two cases:
            \begin{itemize}
            \item If $\min \{p,j-p\} = p$, then Inequality~\eqref{ineqT2T1Case1} becomes
            \begin{equation*}
            \frac{p}{k} + \frac{j-p}{k+p} <  \frac{j}{k} + \frac{p}{j-p},
            \end{equation*}
            which is satisfied for any $p \geq 1$ since $\frac{p}{k} < \frac{p}{j-p}$ (as $j \leq k$ and $j-p > 0$) and $\frac{j-p}{k+p} < \frac{j}{k}$ (as $j-p < j$ and $k+p > k$).
     
            \item If $ \min \{p,j-p\} = j-p$, then Inequality~\eqref{ineqT2T1Case1} becomes
            \begin{equation*}
            \frac{p}{k} + \frac{j-p}{k+p} <  \frac{j}{k} + \frac{j-p}{p},
            \end{equation*}
            which is satisfied for any $p \geq 1$ since $\frac{p}{k} < \frac{j}{k}$ (as $
            0 < p < j$) and $\frac{j-p}{k+p} < \frac{j-p}{p}$ (as $k > 0$).
            \end{itemize}
        This shows that $\st{T'}>\st{T}$ as claimed, contradicting the maximality of $T$. 
        
        While this completes the proof in this case, we remark that $\st{T'} \leq \stBE{n}$ and thus, in particular, $\st{T} < \st{T'} \leq \stBE{n}$. To see this, recall that $T' = \left(\fb{\log k}, \widehat{T}\right)$, where $\widehat{T}$ is a binary tree on $j < n$ leaves. Therefore, by the inductive hypothesis, $\st{\widehat{T}} \leq \stBE{j}$ and thus, $\st{T'} \leq \stBE{n}$. 
        
        \item[Case 2:] Next, assume that $n_1 = k-p$ and $n_2 = j+p$ with $p > 0$ and $k - p \geq j+p$ (note that the second assumption ensures $n_1 \geq n_2$), i.e., 
        $$ T = \left(\BE{k-p}, \BE{j+p} \right).$$
        We now again analyze the two maximal pending subtrees of $T$ and their stairs2 index more in depth, exploiting the fact that they are both binary echelon trees. 
        \begin{itemize}
            \item For $\BE{j+p}$, let $a = 2^{\lfloor \log(j+p) \rfloor}$ denote the largest power of two less or equal to $j+p$, and let $b = j+p-a$. It is possible that $b = 0$, but certainly $a > b$. In particular,
            $$ \BE{j+p} = \left(\fb{\log a}, \BE{b} \right).$$
            \item Similarly, for $\BE{k-p}$, let $c  = 2^{\lfloor \log(k-p) \rfloor}$. We now argue that $c = k/2$. Certainly, $ c < k$, and so $c \leq k/2$. Suppose for contradiction that $c < k/2$. This implies $p > k/2$, and so $2p > k$. However, by assumption, $k-p \geq j+p$, which implies $k \geq j + 2p > 2p$, a contradiction. Thus, $c = k/2$ as claimed. Moreover, $k-p = k/2 + k/2 -p = c + (c-p)$, and thus we can conclude that
            $$ \BE{k-p} = \left( \fb{\log c}, \BE{c-p} \right).$$
        \end{itemize}
        Using Lemma~\ref{l:recursive} and Proposition~\ref{prop:fb}, we thus obtain
        \begin{align}\label{c2t3}
            (n-1) \cdot \st{T} &= (k-p-1) \cdot \stBE{k-p} + (j+p-1) \cdot \stBE{j+p} \nonumber \\
            &= (c-1) + (c-p-1) \cdot \stBE{c-p} + \frac{c-p}{c} \nonumber \\
            &\qquad + (a-1) + (b-1) \cdot \stBE{b} + \frac{b}{a} + \frac{j+p}{k-p}.
        \end{align}
        On the other hand, for the binary echelon tree on $n$ leaves, we have
        \begin{align}\label{BEn}
            (n-1) \cdot \stBE{n} &= \underbrace{(k-1)}_{=(2c-1)} \cdot \underbrace{\st{\fb{\log k}}}_{=1} + (j-1) \cdot \stBE{j} + \frac{j}{k} \nonumber \\
            &= (2c-1) + (j-1) \cdot \stBE{j} + \frac{j}{k},
        \end{align}
        where the last equality follows from the fact that $c = k/2$. 
        In the following, we will subtract Equation~\eqref{c2t3} from Equation~\eqref{BEn} and argue that the result is strictly positive, contradicting the maximality of $T$. We have
        \begin{align}\label{c2diff}
            (n-1) &\cdot \stBE{n} - (n-1) \cdot \st{T} \nonumber \\
            &= (2c-1) + (j-1) \cdot \stBE{j} + \frac{j}{k} - (c-1) - (c-p-1) \cdot \stBE{c-p} - \underbrace{\frac{c-p}{c}}_{= 1 - \frac{p}{c}}\nonumber \\
            &\qquad - (a-1) - (b-1) \cdot \stBE{b} - \frac{b}{a} - \frac{j+p}{k-p} \nonumber \\
            &= \underbrace{\left( (2c-1) - (c-1) - (a-1) - 1 \right)}_{= (c-a)} + \underbrace{\left(\frac{j}{k} - \frac{j+p}{k-p} + \frac{p}{c}  \right)}_{(\ast)} \nonumber \\
            &\qquad + \underbrace{\left( (j-1) \cdot \stBE{j} - (c-p-1) \cdot \stBE{c-p} - (b-1) \cdot \stBE{b} - \frac{b}{a} \right)}_{(\ast \ast)}
        \end{align}
        We now analyze the terms $(\ast)$ and $(\ast \ast)$ in Equation~\eqref{c2diff} further. First, we show that the expression in $(\ast)$ is equal to zero if $k-p = j+p$ and is strictly positive if $k-p > j+p$. To see this, notice that
        \begin{align*}
         \frac{j}{k} - \frac{j+p}{k-p} + \frac{p}{c} \geq 0 &\Leftrightarrow \frac{j}{k} + \frac{p}{c} \geq \frac{j+p}{k-p} \\
         &\Leftrightarrow \frac{j+2p}{k} \geq \frac{j+p}{k-p} \quad \text{ (since $c =k/2$ )} \\
         &\Leftrightarrow (k-p) (j+2p) \geq k(j+p) \\
         &\Leftrightarrow jk + 2kp- jp - 2p^2 \geq kp + jk \\
         &\Leftrightarrow kp \geq jp + 2p^2 \\
         &\Leftrightarrow k \geq j+2p \quad \text{ (since $p > 0$)} \\
         &\Leftrightarrow k- p \geq j + p,
        \end{align*}
        All inequalities are strict if $k-p > j+p$, whereas we have equality everywhere if $k-p = j+p$.

        Now, consider the first grouped term in Equation~\eqref{c2diff}, i.e., consider $c-a$. Since, $k-p \geq j+p$ by assumption, and $\lfloor \log ( \bullet )\rfloor $ is monotonic, we have $c \geq a$, and so $c-a \geq 0$. Moreover, both $a$ and $c$ are powers of two, and thus either $c=a$ or $c \geq 2a$. We now consider these two cases separately:

        \begin{description}
            \item[Case 2a:] First, assume that $c=a$. In this case, in order to complete the proof, we consider the expression in $(\ast \ast)$ further. Substituting $a$ for $c$, we obtain
            \begin{align}
                (j-1) \cdot \stBE{j} -  \left( (a-p-1) \cdot \stBE{a-p} + (b-1) \cdot \stBE{b} + \frac{b}{a} \right)\label{c2atree}
            \end{align}
            Now, consider the following tree:
            $$ \widetilde{T} = \left( \BE{a-p}, \BE{b} \right).$$
            Since $a+b-p = j$, $\widetilde{T}$ is a tree on $j$ leaves. Moreover, since $k-p \geq j+p$ and $a=c$, we must have $c-p \geq b$ and thus $a-p \geq b$. Thus, 
            \begin{align} \label{tildetree}
                (j-1) \cdot \st{\widetilde{T}} &= (a-p-1) \cdot \stBE{a-p} + (b-1) \cdot \stBE{b} + \frac{b}{a-p} \nonumber \\
                &\geq (a-p-1) \cdot \stBE{a-p} + (b-1) \cdot \stBE{b} + \frac{b}{a},
            \end{align}
            where the inequality follows from the fact that $a-p < a$ and $b \geq 0$, and thus $b/a \leq b/(a-p)$. In particular, the inequality is strict if $b > 0$.
            On the other hand, by the inductive hypothesis, $\st{\widetilde{T}} \leq \stBE{j}$. Combining this fact with Equations~\eqref{c2atree} and \eqref{tildetree}, we obtain
            \begin{align*}
               & (j-1) \cdot \stBE{j} -  \left( (a-p-1) \cdot \stBE{a-p} + (b-1) \cdot \stBE{b} + \frac{b}{a} \right) \\
               &\quad \geq (j-1) \cdot \stBE{j} - \left( (a-p-1) \cdot \stBE{a-p} + (b-1) \cdot \stBE{b} + \frac{b}{a-p} \right)\\
               &\quad = (j-1) \cdot \stBE{j} - (j-1) \cdot \st{\widetilde{T}} \\
               &\quad \geq 0.
            \end{align*}
            Here, the first inequality is strict if $b > 0$. Thus, the expression in $(\ast \ast)$ is non-negative for $b \geq 0$ and strictly positive for $b > 0$. 

            Recalling that by assumption $k-p \geq j+p$, we now return to Equation~\eqref{c2diff} and distinguish two cases:
            \begin{itemize}
                \item If $k-p > j + p$, then as shown above, the expression in $(\ast)$ is strictly positive. Since $c-a=0$ and the expression in $(\ast \ast)$ is non-negative (also shown above), the right-hand side of Equation~\eqref{c2diff} is strictly positive, implying that $\st{T} < \stBE{n}$. 
                \item If $k-p = j+p$, we can conclude that neither $k-p$ nor $j+p$ is a power of two. If they were two equal powers of two, then their sum $k-p+j+p=k+j=n$ would also be a power of two, contradicting the assumption that $n$ is not a power of two. This implies that in $j+p=a+b$, we have $b > 0$ (see description of tree $\BE{j+p}$). This in turn implies that the expression in $(\ast \ast)$ is strictly positive, the expression in $(\ast)$ is zero, and $(c-a)=0$, which again shows that the right-hand side of Equation~\eqref{c2diff} is strictly positive, implying that $\st{T} < \stBE{n}$. 
                \end{itemize}

            In summary, in both cases $\st{T} < \stBE{n}$, contradicting the maximality of $T$.
                       
            \item[Case 2b:] Now, assume that $c \geq 2a$. Since $a > b$ by assumption, this in particular implies $c > a+b$. In turn this implies that $c+(c-p)=k-p>j+p=a+b$ (because $c-p>0$) and so term $(\ast)$ in Equation~\eqref{c2diff} is strictly positive. In the following we show that additionally
            \begin{align}
                (c-a) + \underbrace{\left( (j-1) \cdot \stBE{j} - (c-p-1) \cdot \stBE{c-p} - (b-1) \cdot \stBE{b} - \frac{b}{a} \right)}_{(\ast \ast)} > 0, \label{c2bterm}
            \end{align}
            which in summary implies that the right-hand side of Equation~\eqref{c2diff} is strictly positive, implying that $\st{T} < \stBE{n}$ and thus completing the proof.

            Consider the following two trees: $$T_1 = \left(\BE{c}, \BE{j} \right) \text{ and } T_2 = \left( \left(\BE{a},\BE{b} \right), \BE{c-p} \right).$$ Since $a+b=j+p$, both are trees on $c+j$ leaves. Importantly, since $c$ is a power of two, and by assumption $c > a+b > j$, we have that $T_1 = \BE{c+j}$, i.e., $T_1$ is a binary echelon tree. Moreover, since $c > a+b$, $T_2$ must be distinct from the binary echelon tree on $c+j$ leaves. Thus, by the inductive hypothesis, $\st{T_1} > \st{T_2}$. In particular,
            \begin{align*}
                0 &< (c+j-1) \cdot \left( \st{T_1} - \st{T_2} \right) \\
                &= \left( (c-1) + (j-1) \cdot \stBE{j} + \frac{j}{c} \right) \\
                &\qquad - \left( (a-1) + (b-1) \cdot \stBE{b} + \frac{b}{a} + (c-p-1) \cdot \stBE{c-p} + \frac{\min\{j+p,c-p\}}{\max\{j+p,c-p\}}\right) \\
                &= (c-a) + \left( (j-1) \cdot \stBE{j} - (c-p-1) \cdot \stBE{c-p} - (b-1) \cdot \stBE{b} -   \frac{b}{a}\right) \\
                &\qquad + \frac{j}{c} - \frac{\min\{j+p,c-p\}}{\max\{j+p,c-p\}}.
            \end{align*}
            Comparing the last expression to Equation~\eqref{c2bterm}, in order to complete the proof, we need to show that the additional term
            \begin{align*}
                \frac{j}{c} - \frac{\min\{j+p,c-p\}}{\max\{j+p,c-p\}}
            \end{align*}
            is negative. To see this, we distinguish two cases:
            \begin{itemize}
            \item If $\min\{j+p,c-p\} = j+p$ this is clear. More formally, since $c-p \geq j+p > 0$
            \begin{align*}
            \frac{j}{c} - \frac{j+p}{c-p} < 0 &\Leftrightarrow j(c-p) - c(j+p) < 0 \\
            &\Leftrightarrow jc- jp - jc - pc < 0 \\
            &\Leftrightarrow -p(j+c) < 0,
            \end{align*}
            which is a true statement since both $p$ and $j+c$ are strictly positive. In particular, $(j/c)-((j+p)/(c-p) < 0$ as claimed.
            \item If $\min\{j+p,c-p\} = c-p$, we have
            \begin{align*}
             \frac{j}{c} - \frac{c-p}{j+p} < 0 
             &\Leftrightarrow j(j+p) -c(c-p) < 0\\
             &\Leftrightarrow j^2 + jp - c^2 + cp < 0 \\
             &\Leftrightarrow j^2 - c^2 < -p(j+c) \\
             &\Leftrightarrow (j-c)(j+c) < -p(j+c) \\
             &\Leftrightarrow j-c < -p \\
             &\Leftrightarrow c-j > p \\
             &\Leftrightarrow c-p > j.
            \end{align*}
            Now, notice that $c-p \geq 2a-p$. Moreover,
            \begin{align*}
            2a-p > j &\Leftrightarrow 2a > j+p = a+b \\
            &\Leftrightarrow a > b,
            \end{align*}
             which is true by assumption. In particular, $c-p \geq 2a -p > j$, and thus $(j/c)-((c-p)/(j+p) < 0$ as claimed.
            \end{itemize}
        \end{description}
    \end{description}
In summary, we have shown that if $T \neq \BE{n}$, then $T$ does not have maximum stairs2 index. In particular, the binary echelon tree $\BE{n}$ is the \emph{unique} rooted binary tree with $n$ leaves with maximum stairs2 index on $\BT{n}$. This completes the proof.
\end{proof}

\subsection{Two formulas for the maximum value of the stairs2 index} \label{Subsec:MaxValue}
In the previous section, we showed that for any given $n \in \mathbb{N}_{\geq 1}$, the binary echelon tree $\BE{n}$ is the unique rooted binary tree on $n$ leaves with maximum stairs2 index. We now use this fact to obtain formulas for the maximum stairs2 index itself. Thus, let $\st{n}$ denote the maximum value of the stairs2 index on $\BT{n}$, i.e., $\st{n} = \max\limits_{T \in \BT{n}} \st{T} = \stBE{n}$. Exploiting the structure of the binary echelon tree on $n$ leaves and the recursiveness of the stairs2 index (Lemma~\ref{l:recursive}), we directly obtain the following recursion for $\st{n}$.

\begin{cor}\label{cor:recursion}
Let $n \in \mathbb{N}_{\geq 0}$. If $n \in \{0,1\}$, we have $\st{n}=0$. Otherwise, if $n \geq 2$, let $k = 2^{\lfloor \log n \rfloor}$. Then, 
\begin{align*}
    \st{n} &= \frac{1}{n-1} \cdot \left( (k - 1) + (n-k-1) \cdot \st{n-k} + \frac{n-k}{k} \right)
\end{align*}
\end{cor}

\noindent We remark that for even leaf numbers, we can obtain the following alternative simple recursion, which follows from the fact that the binary echelon tree on $2n$ leaves can be obtained from the binary echelon tree on $n$ leaves by replacing each of the $n$ leaves with a cherry, i.e., a pair of leaves with a common parent. 
\begin{align*}
     \st{2n} &= \frac{1}{2n-1} \cdot \left( (n-1) \cdot \st{n} + n \right)
\end{align*}

We now use the recursion of Corollary~\ref{cor:recursion} to derive a closed expression for $\st{n}$.

\begin{cor} \label{cor:sum}
Let $n \in \mathbb{N}_{\geq 0}$. If $n \in \{0,1\}$, we have $\st{n}=0$. Otherwise, if $n \geq 2$ with $n = \sum\limits_{i=1}^\ell 2^{a_i}$ with $\ell \geq 1$ and $a_1, \ldots, < a_\ell \in \mathbb{N}_{\geq 0}$ such that $a_1 < a_2 < \ldots a_\ell$, then
    \begin{align*}
        \st{n} = \frac{1}{n-1} \cdot \left( \sum\limits_{i=1}^\ell \left(2^{a_i}-1 \right) + \sum\limits_{i=1}^{\ell-1} \frac{\sum\limits_{j=1}^{i} 2^{a_j}}{2^{a_{i+1}}} \right).
    \end{align*}
\end{cor}

\begin{proof}
We prove this statement by induction on $n$. For $n \in \{0,1\}$, $\st{n}=0$ by definition. For $n=2$, we have $\st{n}=1$. On the other hand, $2=2^1$, i.e., $\ell=1$ and $a_1=1$. Thus, 
\begin{align*}
    \frac{1}{n-1} \cdot \left( \sum\limits_{i=1}^\ell \left(2^{a_i}-1 \right) + \sum\limits_{i=1}^{\ell-1} \frac{\sum\limits_{j=1}^{i} 2^{a_j}}{2^{a_{i+1}}} \right) &= \frac{1}{1} \cdot \left( (2^1-1) + 0 \right) = 1,
\end{align*}
which establishes the base case. Assume that the claimed formula holds for all $n' < n$. Now, consider $n = \sum\limits_{i=1}^\ell 2^{a_i}$ with $\ell \geq 1$ and $a_1, \ldots, a_\ell \in \mathbb{N}_{\geq 0}$ such that $a_1 < a_2 < \ldots < a_\ell$. We now use the recursion from Corollary~\ref{cor:recursion} to establish the formula for $n$. Therefore, notice that by construction, $k = 2^{\lfloor \log n \rfloor} = 2^{a_\ell}$ and $n-k = \sum\limits_{i=1}^{\ell-1} 2^{a_i}$. Using the inductive hypothesis, we thus obtain
\begin{align*}
    \st{n} &= \frac{1}{n-1} \cdot \left( (k - 1) + (n-k-1) \cdot \st{n-k} + \frac{n-k}{k} \right) \\
    &= \frac{1}{n-1} \cdot \left(  \left(2^{a_\ell}-1 \right) + \left( \sum\limits_{i=1}^{\ell-1} (2^{a_i}-1) + \sum\limits_{i=1}^{\ell-2} \frac{\sum\limits_{j=1}^{i} 2^{a_j}}{2^{a_{i+1}}} \right) + \frac{\sum\limits_{i=1}^{\ell-1} 2^{a_i}}{2^{a_\ell}}\right) \\
    &= \frac{1}{n-1} \cdot \left( \sum\limits_{i=1}^\ell \left(2^{a_i}-1 \right) + \sum\limits_{i=1}^{\ell-1} \frac{\sum\limits_{j=1}^{i} 2^{a_j}}{2^{a_{i+1}}} \right).
\end{align*}
This completes the proof.
\end{proof}

\section{Concluding remarks}\label{Sec:Discussion}
In this paper, we have answered an open question from the literature regarding the maximum value of the stairs2 balance index and the trees that achieve it. We have shown that for all leaf numbers, there is a unique rooted binary tree, namely the binary echelon tree $\BE{n}$, that maximizes the stairs2 index, and have obtained formulas for this maximum value. It is interesting to note that the binary echelon tree is contained in the set of so-called binary weight trees in the sense of \citet{Kersting2021}, and is thus also an extremal tree for the stairs1 index~\cite{Norstrom2012}, the symmetry nodes index~\cite{Kersting2021}, and Roger's $J$ index~\cite{Rogers1996} (see also \cite{Fischer2023}). Importantly, however, the binary echelon tree is generally not the unique most balanced tree for these three indices, which distinguishes the stairs2 index from them. It is also interesting to note that while there are other tree (im)balance indices with a unique most balanced tree such as the total cophenetic index~\cite{Mir2013}, the quadratic Colless index~\cite{Bartoszek2021}, or the rooted quartet index restricted to binary trees~\cite{Coronado2019}, the latter all consider the so-called \emph{maximally balanced tree}~\cite{Mir2013} as most balanced, which generally does not coincide with the binary echelon tree. The stairs2 index thus seems to capture tree balance from a view point distinct from other indices. It would be interesting interesting to analyze the implications of this property for further downstream analyses (like model testing, see, e.g.~\cite{Blum2005}) in future research. 
Additionally, we remark that while we have filled the gap in the literature concerning the maximum value of the stairs2 index, there are still open questions to be explored in future research. As indicated by \citet{Fischer2023}, it is for instance unknown what the expected value and variance of the stairs2 index under the Yule and uniform models for phylogenetic trees are. 
Finally, it would also be interesting to analyze whether the stairs2 index or a generalization of it can be used as a balance index for phylogenetic networks, and if so, what its properties are.

\bibliographystyle{abbrvnat}
\bibliography{References.bib}
 
\end{document}